\documentclass{article}
\usepackage[utf8]{inputenc}
\usepackage{amsmath,amssymb}
\usepackage{times}
\usepackage{array}
\usepackage[english]{babel}
\usepackage{wrapfig} 
\usepackage{amssymb,amsthm,color}
\usepackage[active]{srcltx}
\input epsf
\usepackage{graphicx}
\usepackage{hyperref,pdfsync}
\usepackage{mathrsfs}
\usepackage{changes}
\usepackage{todonotes}

\usepackage{authblk}
\definecolor{wineRed}{rgb}{0.7,0,0.3}

\theoremstyle{definition}
\newtheorem{theorem}{Theorem}[section]

\newtheorem{lemma}[theorem]{Lemma}
\newtheorem{proposition}[theorem]{Proposition}
\newtheorem{corollary}[theorem]{Corollary}

\numberwithin{equation}{section}

\newtheorem{remark}[theorem]{Remark}

\def\sgn{\hbox{sgn}\,}
\def\supp{\hbox{supp}\,}

\def\bN{\mathbb{N}}
\def\bR{\mathbb{R}}

\def\cE{{\mathscr{E}}}

\title{Special Solutions to the Space Fractional Diffusion Problem}
\author[{1}]{Tokinaga Namba {\footnote{namba.rb3.tokinaga@jp.nipponsteel.com}}}
\author[{2}]{Piotr Rybka{\footnote{rybka@mimuw.edu.pl}}}
\author[{3}]{Shoichi Sato{\footnote{shoichi@ms.u-tokyo.ac.jp}}}

\affil[{1}] {{\small Mathematical Science \& Technology Research Laboratory, Advanced Technology Research Laboratories, Research \& Development, Nippon Steel Corporation, 20-1 Shintomi, Futtsu, Chiba Prefecture 293-8511, JP}}
\affil[{2}]
{{\small Institute of Applied
Mathematics and Mechanics\\ University of Warsaw\\ ul. Banacha 2,
02-097 Warsaw, PL}}
\affil[{3}]
{{\small Graduate School of Mathematical Sciences, The University of Tokyo, 3-8-1 Komaba, Meguro-ku, Tokyo, 153-8914 JP}}
\date{\today}

\begin{document}

\maketitle
\begin{abstract}We derive a fundamental solution $\cE$ to a space-fractional diffusion problem on the half-line. The equation involves the Caputo derivative. We establish properties of $\cE$ as well as formulas for solutions to the Dirichlet and Neumann problems in terms of convolution of $\cE$ with data. We also study  integrability of derivative of solutions given in this way. We present conditions sufficient for uniqueness. Finally, we show the infinite speed of signal propagation.
\end{abstract}


\section{Introduction}
In \cite{NaRy} and \cite{KR1} the authors studied the following equation
\begin{equation}\label{eq1}
 \frac{\partial u }{\partial t} - \frac{\partial  }{\partial x} D^\alpha_x u = 0 \qquad (x,t)\in\Omega\times(0,\infty),   
\end{equation}
when $\Omega=  (0,L)$ augmented with the boundary data and initial conditions, here $D^\alpha_x$ is the fractional Caputo derivative with respect to the spacial variable $x$.
The interest in such problems stems from the fact that eq. (\ref{eq1}) appears in the Green–Ampt infiltration  models of subsurface flows, see \cite{VV}. In fact, (\ref{eq1}) is a simplification of free boundary problem, which was studied in \cite{KR2}. 

In this note we derive in Theorem \ref{thm-main} a formula for $\cE$, a self-similar solution to (\ref{eq1})  considered on $(x,t)\in \bR_+$ (here we write $\bR_+$ for $(0,\infty)$). Studying  self-similar solutions or travelling fronts is important, when we  wish to gain insight into the structure of solution and in particular their long time behavior.

Function $\cE$, which we find, is sufficiently smooth to be a classical solution to (\ref{eq1}). We show that
$$
\cE(x,t) = \frac{a_0}{t^{\frac1{1+\alpha}}} E_{\alpha, 1+1/\alpha, 1/\alpha}
\left( - \frac{x^{1+\alpha}}{(1+\alpha)t}\right),
$$
where $E_{\alpha, 1+1/\alpha, 1/\alpha}$
is the 3-parameter generalized Mittag-Leffler function, 
see (\ref{df-ML}). We show that $\cE$ is positive and the integral of $\cE(\cdot, t)$ over $\bR_+$ does not depend on $t>0$. 
In fact, $\cE$ is a fundamental solution of (\ref{eq1}), namely solutions to the Dirichlet and Neumann problems for $\Omega = \bR_+$ 
may be expressed by means of convolution of the initial condition with $\cE$, this is the content of Theorem \ref{prop1}.

The problem of existence of a fundamental solution to various versions of time-fractional problems has already been  addressed in the literature. We name just a few papers dealing with this issues, see \cite{valdinoci}, \cite{mainardi}, \cite{Klekot}, \cite{kim}, \cite{pschu}. The tools used there are different from ours. However, it is not surprising that generalized Mittag-Leffler functions play a role. We could justify positivity of $\cE$ on the grounds of the theory of Mittag-Leffler functions, (actually we do this in the appendix).  However, 
we would like to stress that our proof of positivity of $\cE$ is entirely based on a PDE tool, which the maximum principle, we use this idea after \cite{roscani}.

Let us stress that the justification of formulas for solution based on the convolution requires establishing a number of properties of $\cE(\cdot, 1)$. In particular, we show that $\cE(\cdot, 1)$ is monotone. For this purpose we use a PDE tool, which is the maximum principle. We also show some sort of decay, namely $x\cE(x,1)$ is uniformly bounded, see Lemma \ref{le-zanik}. These properties of $\cE$ seem to be of independent interest.

Once we have a convolution formula for solutions to (\ref{eq1}), we may study properties of solutions to  \eqref{eq1} when $\Omega = \bR_+$. An urging question is about uniqueness of solutions given in this way. We show that if the initial conditions are sufficiently regular, i.e. they are absolutely continuous with compact support, then solutions enjoy sufficient regularity for employing the method of testing the equation with the solution itself, see Proposition \ref{pr-uq}. This technique immediately yields uniqueness and decay of solutions.

Eq. (\ref{eq1}) contains a parameter $\alpha$, so does $\cE$. Due to analyticity of $E_{\alpha, 1+1/\alpha, 1/\alpha}$ we deduce that if $u^\alpha$ is a solution to (\ref{eq1}) on $\bR_+$, then $u^\alpha \to u^{\alpha_0}$ in $L^1(\bR_+)$ when $\alpha$ goes to $\alpha_0\in (0,1]$.  We explicitly exclude $\alpha_0 =0$, which is due to the fact that $E_{\alpha, 1+1/\alpha, 1/\alpha}$
has no limit on $\bR_+$ when $\alpha\to 0$. However, this case is covered by 
\cite[Theorem 6.1]{NaRy}. 

With this observation we may 
address here the
issue of the speed of the signal propagation. This is a bit puzzling because for $\alpha =1$, eq. (\ref{eq1})  becomes the heat equation with the infinite speed of propagation, while for $\alpha =0$ problem (\ref{eq1}) is the transport  eq. for which the speed is finite.


Our numerical experiments presented in \cite{NRV} show that an initial pulse moves to the left with a finite speed. The same conclusions are drawn on the basis of numerical simulations by the authors of \cite{luczko} who dealt with the time-fractional diffusion-wave equation. However, our Proposition \ref{pi} stated for \eqref{eq1} with $\Omega=\bR_+$ and the Neumann boundary condition shows that actually the speed of the signal is infinite -- the support of the solution instantly becomes equal to $[0,\infty)$. This is shown
with the help of the explicit formulas employing the fundamental solution, $\cE$ constructed here. This fact does not contradict numerical observation. 

After presenting the content we describe the organization of the note. In Section 2 we recall the fundamentals of the fractional calculus. Section 3 is devoted to the  derivation of   a formula for a self-similar solution $\cE$. Here we also study its properties collected in Theorem \ref{thm-main} and we derive the formulas for the integral representation  of unique solutions. In  the Appendix we present the derivation of $\cE$ which based on the properties of the generalized  Mittag-Leffler function $E_{\beta k l}$ and not the  series manipulation. We also show a short proof of positivity of $\cE$, which follows from the theory of the  generalized  Mittag-Leffler function.

\section{Preliminaries} 

We will simultaneously recall the definitions of the Caputo and Riemann-Liouville fractional derivatives. For a function $f\in L^1(0,L)$ and $\alpha\in(0,1)$ we introduce the fractional integration operator by setting,
\begin{equation}\label{frak1}
    (I^\alpha f)(x) = \frac1{\Gamma(\alpha)}\int_0^x (x-z)^{\alpha -1}f(z)\,dz.
\end{equation}
For an absolutely continuous function $u\in AC[0,L]$  we define the Caputo fractional derivative of order $\alpha\in (0,1)$ by the following formula,
\begin{equation}\label{frak2}
  D^\alpha_C u (x)=  (I^{1-\alpha}u')(x) = 
  \frac1{\Gamma(1-\alpha)}\int_0^x \frac{u'(s)}{(x-s)^\alpha}\,ds,
\end{equation}
while the Riemann-Liouville fractional derivative has the form,
\begin{equation}\label{frak2}
  D^\alpha_{RL} u =  \frac{d}{dx} (I^{1-\alpha}u) .
\end{equation}
Later, for the sake of simplicity of notation we will write $D^\alpha$ for $D^\alpha_C $, occasionally with a lower index $x$ indicating the space variable.

We notice that if $u$ is absolutely continuous and $u(0)=0,$ then we have
\begin{equation}\label{rn6}
\frac d{d x}\left(I^{1-\alpha} u\right)= I^{1-\alpha} \frac d{dx}u \qquad\hbox{i.e.}\qquad D^\alpha_{C}u = D^\alpha_{{RL}}u.
\end{equation}
The following formula explains the relationship between the two types of derivatives for a  general  function $u\in AC[0,L]$,
\begin{equation}\label{RL-C}
D^\alpha_{x_{RL}}u = D^\alpha_{x_C}u + \frac{x^{-\alpha}}{\Gamma(1-\alpha)} u(0).
\end{equation}
The fractional integration is the inverse of the Caputo derivative up to a constant,
\begin{equation}\label{cal-roz}
    I^\alpha D^\alpha_{C}u(x) = u(x) - u(0).
\end{equation}

In our analysis we will need a more convenient representation of the operator  $(D^\alpha_x u)_x$. 
For this purpose we need to recall:
\begin{lemma}\label{altdef} (see \cite[Proposition 2.1]{NaRy})
Let $u:[0,l)\to\mathbb{R}$ be such that $u\in C^2(0,l)\cap C[0,l)$ and $u'\in L^1(0,l)$.
Then, $(D^\alpha_x u)_x$ exists everywhere in $(0,l)$ and
\begin{equation}\label{e:another}
\begin{split}
(D^\alpha u)_x(x)=
\frac{1}{\Gamma(1-\alpha)}&\left(\frac{\alpha(u(0)-u(x))+(\alpha+1)u'(x)x}{x^{\alpha+1}}\right.\\
&\quad\left.+\alpha(\alpha+1)\int_0^x[u(x-z)-u(x)+u'(x)z]\frac{dz}{z^{\alpha+2}}\right)
\end{split}
\end{equation}
for $x\in(0,l)$.
\end{lemma}

With the help of this Lemma we will compute action of $(D^\alpha u)_x$ on scaled functions:
\begin{corollary}\label{co-ald}
Let $u:[0,l)\to\mathbb{R}$ be such that $u\in C^2(0,l)\cap C[0,l)$ and $u'\in L^1(0,l)$. If $\lambda>0$ and we set $v_\lambda(x) = u(\lambda^{\frac1{1+\alpha}} x)$, then $(D^\alpha_x v_\lambda)_x(x) = \lambda (D^\alpha_x u)_y(\lambda^{\frac1{1+\alpha}} x)$.
\end{corollary}
\begin{proof}
We use Lemma \ref{altdef}  to calculate $(D_x^\alpha v_\lambda)_x$, 
\begin{align*}
&(D_x^\alpha v_\lambda)_x(x)\\
&=\frac{1}{\Gamma(1-\alpha)}\left(\frac{\alpha(v_\lambda(0)-v_\lambda(x))+(\alpha+1)(v_\lambda)_x(x)x}{x^{\alpha+1}} 
+\alpha(\alpha+1)\int_0^x[v_\lambda(x-z)-v_\lambda(x)+(v_\lambda)_x(x)z]\frac{dz}{z^{\alpha+2}} \right)\\
&=\frac{1}{\Gamma(1-\alpha)}
\left(\frac{\alpha(u(0) - u(\lambda^{\frac{1}{1+\alpha}}x))+(\alpha+1)\lambda^{\frac{1}{1+\alpha}}u_y(\lambda^{\frac{1}{1+\alpha}}x)x}{x^{\alpha+1}}\right.\\
&\quad \left.+\alpha(\alpha+1)\int_0^x[u(\lambda^{\frac{1}{1+\alpha}}(x-z))- u(\lambda^{\frac{1}{1+\alpha}}x)+\lambda^{\frac{1}{1+\alpha}}u_y(\lambda^{\frac{1}{1+\alpha}}x)z]\frac{dz}{z^{\alpha+2}}\right)\\
&=\frac{1}{\Gamma(1-\alpha)}\left(\frac{\alpha(u(0)-u(\lambda^{\frac{1}{1+\alpha}}x))+(\alpha+1)u_y(\lambda^{\frac{1}{1+\alpha}}x)(\lambda^{\frac{1}{1+\alpha}}x)}{\lambda^{-1}(\lambda^{\frac{1}{1+\alpha}}x)^{\alpha+1}}\right.\\
&\quad\left.+\alpha(\alpha+1)\int_0^x[u(\lambda^{\frac{1}{1+\alpha}}(x-z))-u(\lambda^{\frac{1}{1+\alpha}}x)+u_y(\lambda^{\frac{1}{1+\alpha}}x)(\lambda^{\frac{1}{1+\alpha}}z)]\frac{dz}{z^{\alpha+2}}\right).
\end{align*}
Changing the variable of integration by $\lambda^{\frac{1}{1+\alpha}}z=\xi$ we see $dz/z^{\alpha+2}=\lambda d\xi/\xi^{2+\alpha}$. Hence that the last term of the right-hand-side (RHS) is equal to
$$ 
\frac{\alpha(\alpha+1)\lambda}{\Gamma(1-\alpha)} \int_0^{\lambda^{\frac{1}{1+\alpha}}x}[u(\lambda^{\frac{1}{1+\alpha}}x-\xi)-u(\lambda^{\frac{1}{1+\alpha}}x)+u_y(\lambda^{\frac{1}{1+\alpha}}x)\xi]\frac{d\xi}{\xi^{\alpha+2}}.
$$
Therefore applying Lemma \ref{altdef} again yields
$$
(D_x^\alpha v_\lambda)_x(x)=\lambda^{1}(D_y^\alpha u)_y(\lambda^{\frac{1}{1+\alpha}}x).
$$ 
\end{proof}

In our construction of the self-similar solution we will use a three-parameter  generalized Mittag-Leffler function $E_{\beta,m,l}$.
It is defined by the following series for $z\in\mathbb{C}$, (see \cite[formula (1.9.19)]{kilbas})
\begin{equation}\label{df-ML}
    E_{\beta,m,l}(z) = \sum_{n=0}^\infty c_n^{\beta ml} z^n,
    \qquad \Re \beta>0, \ m>0, \ -\beta(jm + l)\not\in \bN\setminus\{0\},
\end{equation}
where
$$
c_n^{\beta ml} =\prod_{i=0}^{n-1} \frac{\Gamma(\beta(im +l)+1)}{\Gamma(\beta(im+l+1)+1)}.
$$
It is worth recalling that due to 
\cite[Theorem 1]{goro}, see also page 48 in \cite{kilbas}, we know that $E_{\beta,m,l}$ is an entire function of order $(\Re \beta)^{-1}$ and type $m^{-1}$, i.e.
$$
|E_{\beta,m,l}(z)|< \exp( (m^{-1}+ \epsilon) |z|^{1/\beta})
$$
holds for $|z| \ge r_0(\epsilon),$ where $r_0(\epsilon)$ is  sufficiently large.

In several places we will use the following observation concerning the Mittag-Leffler functions. Let us set,
\begin{equation}\label{def-phi} 
\Phi(x) = E_{\alpha, 1+ 1/\alpha, 1/\alpha} \left(- \frac{x^{1+\alpha}}{1+\alpha}\right).
\end{equation}
We recall:
\begin{proposition}\label{rnieE} (see \cite[Example 4.11]{kilbas},  \cite[Theorem 4]{saigo})
Function $\Phi$ defined in (\ref{def-phi}) satisfies the following fractional ordinary differential equation,
\begin{equation}\label{ode}
 D_{y_C}^\alpha v(y)+\frac{1}{1+\alpha}yv(y)=0, \quad y>0, \qquad v(0) = 1. 
\end{equation}
\end{proposition}

\section{The fundamental solution and its properties}
We will derive a 
self-similar solution, $\cE,$ to the following equations,
\begin{equation}\label{rn1}
\frac{\partial u }{\partial t} - \frac{\partial  }{\partial x} D^\alpha u = 0 \qquad (x,t)\in (0,\infty)^2.
\end{equation}
This is done in the theorem below,
where we also present the basic  properties of $\cE$.
\begin{theorem}\label{thm-main}
Function $\cE: (0,\infty)^2 \to \bR$ given by formula  
\begin{equation}\label{defE}
    \cE(x,t) = a_0 t^{-\frac{1}{1+\alpha}} 
    E_{\alpha, 1+1/\alpha, 1/\alpha}\left( -\frac{x^{1+\alpha}}{ (1+\alpha)t}\right).
\end{equation}
where 
\begin{equation}\label{df-a0}
\frac{1}{a_0} = 2\int_0^\infty E_{\alpha, 1+1/\alpha, 1/\alpha}( -x^{1+\alpha}/ (1+\alpha))\,dx
\end{equation}
is well-defined:\\  
(1) $\cE\in C^2((0,\infty)^2)$ and $\cE$   is a self-similar solution to (\ref{rn1}), i.e. it is invariant under the transformation $(x,t) \mapsto (\lambda^{\frac1{1+\alpha}} x, \lambda t)$ for $\lambda>0$.\\
(2)  $\cE$ is positive for all $x,t >0$.\\
(3) For all  $t>0$ function $\cE(\cdot,t)$ is decreasing.\\
(4) For all  $t>0$ function $\cE(\cdot,t)$ is in $L^1(0, \infty)$ and
the integral $\int_0^\infty \cE(x,t)\,dx=\frac 12$.
\end{theorem}
\begin{remark}
Using $\Phi$ defined in (\ref{def-phi}) we can write $\cE$ shortly as follows,
$$
\cE(x,t) = a_0 t^{-\frac1{1+\alpha}} \Phi\left(\frac x{t^{\frac1{1+\alpha}}}\right).
$$
\end{remark}
Our proof of Theorem \ref{thm-main} will be conducted in a number of steps. Here is the first one:
\begin{lemma}\label{le-sss}
Let us suppose that $u\in C^2((0,\infty)^2 )$. Then $u$ is a solution to (\ref{rn1}) if and only if $u_\lambda$ is a solution, where
$$
u_\lambda(x,t) = u(\lambda^{\frac1{1+\alpha}}x, \lambda t) .
$$
\end{lemma}
\begin{proof}
Now, it is straightforward to see that
$$
(u_\lambda)_t(x,t)=\lambda^{1}u_s(\lambda^{\frac{1}{1+\alpha}}x,\lambda t).
$$
We use Corollary \ref{co-ald} to compute $(D^\alpha u_\lambda)_x$. We apply this result to $v_\lambda$ defined as $v_\lambda(x) = u(\lambda^{\frac1{1+\alpha}}, \lambda t)$. This leads us to the identity,
$$
(D^\alpha u_\lambda)_x(x,t) = (D^\alpha v_\lambda)_x (x)= \lambda^1(D^\alpha u)_y(\lambda^{\frac1{1+\alpha}}, \lambda t).
$$
Thus, we see that
$u$ satisfies \eqref{rn1} if and only if 
$u_\lambda$ fulfills 
$$
(u_\lambda)_t(x,t)-(D_x^\alpha u_\lambda)_x(x,t)=0
$$
for all $x>0,\ t>0$.
\end{proof}
The above Lemma tells us that self-similar solutions depend only on $\xi = x t^{-\frac1{1+\alpha}}$. However, if we want to obtain a solution, whose average over $\bR_+$ is independent of time, then we must consider $\cE$ of the form
\begin{equation}\label{rnE}
\cE(x,t) = a_0 t^{\gamma} v( x^{1+\alpha}/ t),    
\end{equation}
where $v$ is integrable. It is easy to see that for 
$$
\gamma = - \frac 1{1+\alpha}
$$
and $t_2>t_1>0$ we have
$$
\int_0^\infty \cE(x,t_2)\,dx = \int_0^\infty \cE(x,t_2)\,dx.
$$
We are now ready to derive the form of $\cE$.
\begin{lemma}\label{le-d}
Let us assume that $v$ appearing in \eqref{rnE} is analytic and $v(0) =1$. Then, 
$$
v(z) = E_{\alpha, 1+1/\alpha, 1/\alpha}
\left(- \frac{z}{1+\alpha}\right),
$$
where $E_{\beta,m,l}$ is the generalized Mittag-Leffler function defined in (\ref{df-ML}).
\end{lemma}
\begin{proof}
Let us suppose that
$$
v(z) = \sum_{n=0}^\infty c_n z^n.
$$
Then, inserting $\cE$ defined by (\ref{rnE}) into (\ref{rn1}) yields
$$
\sum_{n=0}^\infty c_n (\gamma -n) \frac{x^{(1+\alpha)n}}{t^{n-\gamma +1}} = 
\sum_{n=1}^\infty c_n 
\frac{\Gamma((1+\alpha)n +1)[(1+\alpha)(n-1) +1]}{\Gamma((1+\alpha)(n-1)+2)}
\frac{x^{(1+\alpha)(n-1)}}{t^{n-\gamma}} ,
$$
where we took into account that $D^\alpha x^\beta = \frac{\Gamma(\beta+1)}{\Gamma(\beta+1-\alpha)} x^{\beta-\alpha}.$ Hence, we find the formula for $c_n$,
$$
c_{n} = \frac{(-1)^{n}}{(1+\alpha)^{n}} b_n,
$$
where we set $c_0=1$ and
\begin{equation}\label{defb}
    b_n = \prod_{i=0}^{n-1} \frac{\Gamma(\alpha i + i +2)}{\Gamma(\alpha(i+1) + i +2)}.
\end{equation}
Hence, 
$$
v(z) = \sum_{n=0}^\infty (-1)^n b_n \left( \frac z{1+\alpha}\right)^n.
$$
If we take into account the form of $b_n$'s, then we realize that
$$
v(z) = E_{\alpha, 1+1/\alpha, 1/\alpha}\left(- \frac z{1+\alpha}\right),
$$
where $E_{\beta,m,l}$ is a generalized Mittag-Leffler function defined in (\ref{df-ML}).
Thus, we reache  $\cE$  of the  form (\ref{defE}), where the multiplicative constant $a_0$ has to be determined by other means.
\end{proof}
\begin{remark}
The argument above is based on the series manipulation. In the Appendix, we present the proof of this lemma, which is based only on the theory of the generalized Mittag-Leffler function $E_{\beta, m, l}$.
\end{remark}

An important step in our analysis is checking that  $\cE$ is indeed positive. The argument we use is based on the use of the maximum principle for (\ref{eq1}). This  idea was used first in the proof of a similar result in \cite{roscani}. We provide our own and extended version of the argument. In the Appendix we present a short proof of the same result, which is of independent interest. It is based on the theory of the three-parameter Mittag-Leffler function.
\begin{lemma}\label{le-posit}
Function 
$\Phi$  defined in (\ref{def-phi}) is positive for all $x>0$.
\end{lemma}
\begin{proof}
Let us set
$$
u(x,t) 
:= \int_0^{xt^{-\frac1{1+\alpha}}} \Phi(z) \, dx.
$$
We will see that 
\begin{equation}\label{nier}
    \left(\frac\partial{\partial t} - \frac{\partial }{\partial x}D^\alpha_{x_C} \right) u<0.
\end{equation}
Indeed, 
$\frac{\partial u}{\partial t}(x,t) = - \frac{1}{t(1+\alpha)}y \Phi(y)$,
where $y = x {t^{-\frac 1{1+\alpha}}}$ 
and by Corollary \ref{co-ald} we see that
$\frac{\partial }{\partial x}D^\alpha_{x_C}u(x,t) = 
\frac 1t \frac{\partial }{\partial y}D^\alpha_{y_C}u(y,t)$. Thus, due to 
$$
\frac{\partial }{\partial x}D^\alpha_{x_C} u = 
D^\alpha_{x_{RL}} \frac{\partial }{\partial x}u 
$$
and (\ref{RL-C}) we obtain
$$
\left(\frac\partial{\partial t} - \frac{\partial }{\partial x}D^\alpha_{x_C} \right) u(x,t)= - \frac 1t 
\left(\frac{y\Phi(y)}{1+\alpha}+ D^\alpha_{y_{RL}}\Phi(y)\right) = 
- \frac 1t 
\left(\frac{y\Phi(y)}{1+\alpha}+ D^\alpha_{y_{C}}\Phi(y)\right)
- \frac {y^{-\alpha}}{t\Gamma(1-\alpha)} .
$$
Now, we invoke Proposition \ref{rnieE} to conclude that the expression in parenthesis vanishes. Thus, (\ref{nier}) follows.


Let us suppose our claim is not valid and the set
$$
A =\{x>0: \Phi(x) 
<0\}
$$
is not empty. Since $\Phi(0) =1$, 
we see that $\inf A =: x_0 >0$. Due to the continuity of $\Phi$
there is $x_1> x_0$ such that 
$$
u(x, 1) 
>0\qquad\hbox{for } x\in (0,x_1).
$$
Let us set 
\begin{equation}\label{def-om}
\Omega =\{(x,t)\in (0,\infty)^2:\ x\in (0 , x_1 t^{-1/(1+\alpha)}), t\in (1, 2)\}  .  
\end{equation}
By the weak maximum principle for (\ref{nier}) in non-cylindrical regions, see  \cite[Lemma 8]{KR2},
we have
$$
\sup\{ u(x,t): (x,t) \in \Omega\} = \max\{ u(x,t): (x,t)\in \partial\Omega \setminus (\{2\}\times \bR_+)\} =: M.
$$
However, since $u(0, t) =0$ and $u(x_1 t^{-1/(1+\alpha)}, t)= u(x_1,1)< u(x_0, 1) $ for $t\in[1,2]$, then we see that
$$
M = u(x_0, 1).
$$
If we had the strong maximum principle, then we would have reached a contradiction. But we do not have it, thus, we have to continue our argument. For a positive $\epsilon$ and $(x,t) \in \bR_+ \times(1,\infty)$ we set,
$$
v_\epsilon(x,t) = \epsilon c_\alpha x^{1+\alpha} + \epsilon(t -1),
$$
where $c_\alpha = \frac 1{(1+\alpha)\Gamma(1+\alpha)}.$  Due to the choice of $c_\alpha$ we see that $\frac{\partial v_\epsilon }{\partial t} - \frac{\partial  }{\partial x} D^\alpha v_\epsilon = 0$. Thus, the sum $u + v_\epsilon$ satisfies (\ref{nier}) and we may apply the weak maximum principle, see  \cite[Lemma 8]{KR2}, which yields,
\begin{equation}\label{maxp}
    \sup\{ (u+ v_\epsilon)(x,t): (x,t) \in \Omega\} = \max\{ (u+ v_\epsilon)(x,t): (x,t)\in \partial\Omega \setminus (\{2\}\times \bR_+)\} =: M_\epsilon.
\end{equation}
We want to select such $\epsilon$ that $M_\epsilon = \max (u+v_\epsilon)(x,1)$.
We notice that $(u+ v_\epsilon)(0,t)= \epsilon(t-1)\le \epsilon$ and we may restrict $\epsilon$ so that $\epsilon< u(x_1, 1)$. We take even smaller $\epsilon$ so that 
$$
\frac{\partial}{\partial x}(u+ v_\epsilon)(x_1,t)<0.
$$
Hence, there is $x_2\in (x_0, x_1)$ such that 
$$
\max\{ (u+ v_\epsilon)(x,1): x \in (0, x_1)\} = (u+ v_\epsilon)(x_2,1)
$$
We want to guarantee that $M_\epsilon$ equals $(u+ v_\epsilon)(x_2,1) = u(x_2, 1) + \epsilon c_\alpha x_2^{1+\alpha}.$

We look at $u+v_\epsilon$ on the last part of the parabolic boundary of $\Omega$, 
$$
\gamma = \{ (x_1 t^{1/(1+\alpha)},t ):\ t\in (1, 2)\}.
$$
For $(x,t)\in \gamma$ we have
$$
(u+v_\epsilon)(x_1 t^{1/(1+\alpha)},t) = u(x_1, 1) + \epsilon (c_\alpha x_1^{1+\alpha}t + t -1) = u(x_1, 1) + \epsilon c_\alpha x_1^{1+\alpha} +(t -1)\epsilon (c_\alpha x_1^{1+\alpha} + 1).
$$
We see that right-hand-side above is greater than $(u+v_\epsilon)(x_1, 1)$ for $t\in (1,2)$. At the same time 
we see that the left-hand-side attains its maximum for $t=2$ and we may take $\epsilon$ so small that
$$
\epsilon (c_\alpha x_1^{1+\alpha} + 1) + (u+v_\epsilon)(x_1, 1) < (u+ v_\epsilon)(x_2,1) = M_\epsilon.
$$
Now, we shall see that $\sup_\Omega (u+ v_\epsilon) > M_\epsilon$. Let us consider points $(x,t)$ of the following form, $(x_2 t^{1/(1+\alpha)}, t)$, $t\in (1,2).$ Now, we compute
values of $(u+ v_\epsilon)$ there for $t>1$. We obtain,
\begin{eqnarray*}
(u+ v_\epsilon)(x_2 t^{1/(1+\alpha)}, t) &= &u(x_2, 1) + \epsilon c_\alpha x_2^{1+\alpha}t
+ \epsilon(t-1) = u(x_2, 1) + \epsilon c_\alpha x_2^{1+\alpha} + \epsilon (t-1) (c_\alpha x_1^{1+\alpha} + 1) \\
&> &(u+ v_\epsilon)(x_2, 1)= M_\epsilon.
\end{eqnarray*}
But this inequality violates (\ref{maxp}). Thus, our claim follows.
\end{proof}

In fact, in the course of proof of Lemma \ref{le-posit}, we established the following fact:
\begin{lemma}
If $U\in C^{1+\alpha}([0,\infty))$ and $u(x,t) = U(x t^{-\frac1{1+\alpha}})$ satisfies inequality (\ref{nier}) in $\Omega$ defined in (\ref{def-om}), where $x_1>0$ is now arbitrary, then $u(x,1) \equiv U(x)$ cannot attain a maximum inside $(0, x_1)$.
\end{lemma}
We notice that the same argument, after the necessary obvious changes,  yields:
\begin{lemma}\label{le-pom}
If $U\in C^{1+\alpha}([0,\infty))$ and $u(x,t) = U(x t^{-\frac1{1+\alpha}})$ satisfies inequality 
\begin{equation}\label{2nier}
 \left(\frac\partial{\partial t} - \frac{\partial }{\partial x}D^\alpha_{x_C} \right) u>0
\end{equation}
in $\Omega$ defined in (\ref{def-om}), where $x_1>0$ is now arbitrary, then $u(x,1) \equiv U(x)$ cannot attain a minimum inside $(0, x_1)$. \qed
\end{lemma}
This observation immediately implies part (3) of Theorem \ref{thm-main}.
\begin{lemma} \label{le-dec}
For all  $t>0$ function $\cE(\cdot,t)$ is decreasing.
\end{lemma}
\begin{proof}
Since $\cE$ is a self-similar solution, we may restrict our attention to $t=1.$ Let us suppose our claim is false and $\cE(\cdot,1)$ attains a minimum at $x_0$. We can find $x_1>x_0$ such that $\cE(x_1,1)> \cE(x_0,1)$. Now, we choose $\epsilon>0$ sufficiently small that 
$$
v_\epsilon(x,t) = \cE(\cdot,t) - \epsilon \frac x{t^{\frac 1{1+\alpha}}},
$$
so that $v_\epsilon(\cdot, 1)$ attains its minimum at $x_2< x_0$. In particular this implies that $v_\epsilon(x_2, 1) < v_\epsilon(x_1, 1).$ Moreover, for all $\epsilon>0$ inequality (\ref{2nier}) is satisfied. As a result we may apply Lemma \ref{le-pom} to $v_\epsilon(\cdot , 1)$ to deduce that $v_\epsilon(\cdot , 1)$ cannot attain any minimum.

We observe that for no positive $x_0$ function $\Phi$ is increasing on $(0, x_0).$ If  such a point existed, then $D^\alpha \Phi \ge 0$ on $(0, x_0)$, but this contradicts (\ref{ode}) due to positivity of $\Phi.$ Hence, $v_\epsilon$ cannot attain any maximum in $(0,\infty)$.

These observations imply that $v_\epsilon(\cdot , 1)$ is decreasing. Indeed, if $0<x_1 <x_2$ and $v_\epsilon(x_1, 1) < v_\epsilon( x_2, 1)$. Then, taking into account that $v_\epsilon(0, 1)=1> v_\epsilon(x_2,1)$ we deduce that $v_\epsilon(\cdot , 1)$ must attain a minimum in the interval $(0,x_2)$ but this is impossible. Hence, the claim follows.

Finally, we notice that $\cE(x,1) = \lim_{\epsilon\to 0^+} v_\epsilon(x, 1))$, which implies monotonicity of $\cE(\cdot, 1).$
\end{proof}
 
Now, we will show that $\cE$ is integrable over the positive half-line.
\begin{lemma}\label{le-int}
Function $\Phi$  defined in (\ref{def-phi})
is bounded with a bound uniform in $\alpha$ and it is integrable over $(0,\infty).$
\end{lemma}
\begin{proof}
We will show first the boundedness of $\Phi$. 
Due to Proposition   \ref{rnieE} 
and Lemma \ref{le-posit}, we notice,
$$
D_{C}^\alpha \Phi(x)
=  - \frac{x}{1+\alpha} \Phi(x) <0.
$$
We may apply the fractional integration operator $I^\alpha$ to both sides of the above inequality. Due to (\ref{cal-roz}) we obtain,
\begin{equation}\label{r-oszac}
\Phi(x) - \Phi(0)
= 
I^\alpha D_C^\alpha \Phi(x) <0.
\end{equation}
Hence,
\begin{equation}\label{estE}
\Phi(x) \equiv E_{\alpha, 1+1/\alpha, 1/\alpha}(- \frac{x^{1+\alpha}}{1+\alpha}) 
<  \Phi(0) = 1.
\end{equation}
Let us stress that the estimate (\ref{estE}) is uniform in $\alpha.$

Now, we shall see that boundedness of $\Phi$ implies its integrability. For this purpose
we rewrite (\ref{ode}) 
using (\ref{RL-C}) as follows,
$$
\frac{1}{1+\alpha}\Phi(x) = \frac{x^{-1-\alpha}}{\Gamma(1-\alpha)}
- \frac {x^{-1}}{\Gamma(1-\alpha)}\frac d{dx}\int_0^x\frac{\Phi(t)}{(x-t)^{1-\alpha}}\,dt.
$$
We integrate it over $[1,R]$ and we reach,
$$
\frac{\Gamma(1-\alpha)}{1+\alpha} \int_1^R \Phi(s)\,ds \le
\int_1^R x^{-1-\alpha}\,dx +
\left|  \int_1^R x^{-1} \frac d{dx}\int_0^x\frac{\Phi(t)}{(x-t)^{1-\alpha}}\,dtdx
\right| = J_1 + |J_2|.
$$
In the second term we integrate by parts. This yields,
$$
 J_2 = \int_1^R x^{-2} \int_0^x\frac{\Phi(t)}{(x-t)^{1-\alpha}}\,dt dx
+ \left. 
 x^{-1} \int_0^x\frac{\Phi(t)}{(x-t)^{1-\alpha}}\,dt
 \right|_{x=1}^{x=R}. 
$$
Now, we use (\ref{estE}) and positivity of $\Phi$ to see that
$$
 J_2  \le \frac{1 - R^{\alpha-1}}{\alpha(1-\alpha)}
+ \frac{R^\alpha}{\alpha R} < \frac{1}{\alpha(1-\alpha)}. 
$$
If we combine it with an easy estimate on $J_1$ we arrive at
$$
\int_1^R \Phi(s)\,ds \le  \frac{(2-\alpha)(1+\alpha)}{\alpha(1-\alpha)\Gamma(1-\alpha)}.
$$
Our claim follows.
\end{proof}
We notice that the  estimate for the integral of $\Phi$ blows up at $\alpha= 1$ and $\alpha =0$ which shows our method is not optimal, because the fundamental solution to the heat equation is integrable due to its exponential decay at infinity.

\bigskip
We are now ready to finish {\it the proof of Theorem \ref{thm-main}}. The derivation is performed in Lemmas \ref{le-sss} and \ref{le-d}. The properties of $\Phi$
were established in
Lemmas \ref{le-posit}, \ref{le-dec} and \ref{le-int}. In particular, they guarantee that the integral
$$
\frac{1}{a_0} = 2\int_0^\infty E_{\alpha, 1+1/\alpha, 1/\alpha}( -x^{1+\alpha}/ (1+\alpha))\,dx
$$
is finite and positive. Hence, the definition of $a_0$ given in (\ref{df-a0}) is correct and $\cE$ is well-defined with the properties we stated. \qed

\begin{remark}
Actually, our proof   shows  that  for all $t>0$ we have $\cE(\cdot, t)\in C^{1+\alpha}([0,\infty))$ and this regularity is optimal.
\end{remark}

\bigskip
We constructed $\cE$ on $\bR_+^2$. In order  to discuss its properties leading to a justification of the name 'fundamental solution' we have to extend $\cE$  to $\bR_+$, without changing the notation. Actually, the function $x^{1+\alpha}$ is naturally defined for negative argument as $|x|^{1+\alpha}$. We also set  $\cE$ equal to  zero on $\bR \times (-\infty, 0]$, so finally $$\cE(x,t) = \cE(|x|,t) \chi_{\bR_+}(t).$$
We do not
want to discuss the action of $D^\alpha_{x_C}$ on $\mathscr{D}(\bR)$ so we will not try to show that $(\frac \partial{\partial t} - \frac \partial{\partial x}D^\alpha_{x_C}) \cE = \delta_0$. Instead we will justify the representation formulas for solution to (\ref{rn1}) augmented with the initial and boundary data. In fact, we will reuse the well-known formula derived with the help of the reflection principles for solutions to the heat equation on the half line.

\begin{theorem}\label{prop1}
Let us suppose that  $g\in L^p(0,\infty),$ where  $p\in[1,\infty)$ and $g$ has compact support, (resp. $g\in C^0_c([0,\infty))$
and   we set 
$$
\cE_t(x) = \cE(x,t).
$$
We define functions $w_1$, $w_2$ by the following formulas,
\begin{equation}\label{rn3.5}
w_1(x,t) = \int_0^\infty ( \cE(x-y,t) - \cE(x+y, t))g(y)\,dy,
\end{equation}
\begin{equation}\label{rn3.6}
w_2(x,t)  = \int_0^\infty ( \cE(x-y,t) + \cE(x+y, t))g(y)\,dy.
\end{equation}
Then, \\
(a) For all $t, R>0$ functions $w_1(\cdot,t)$ and $w_2(\cdot,t)$ belong to $C^{1+\alpha}([0,R])$, they are classical solutions to 
\begin{equation}\label{rn3}
\begin{array}{ll}\displaystyle{
\frac{\partial w }{\partial t} - \frac{\partial  }{\partial x} D^\alpha w = 0}      & x,t >0, \\
 ~ & ~\\
w(x,0) = g(x)     & x>0.
\end{array}
\end{equation}
In addition $w_1$ (resp. $w_2$) satisfies the Dirichlet (resp. Neumann) boundary condition,
\begin{equation}\label{r12.5}
    w(0,t) = 0,\qquad (\hbox{resp. } w_x(0,t)=0)\qquad \hbox{for }t>0.
\end{equation}
(b) Functions $w_1$ and $w_2$ belong to $L^\infty(\bR_+; L^p(\bR_+))$ if $p<\infty$ (resp. $w_1, w_2 \in L^\infty(\bR_+; C(\bR_+))$, when $g\in C^0_c([0,\infty))$).
The initial condition is satisfied in the sense below. However, when $g$ is continuous, we require $g(0) =0$ in case of the Dirichlet data,
\begin{equation}\label{r-ic}
 \lim_{t\to 0}\|w_1(\cdot, t) - g\|_{L^p} =0, \qquad(\hbox{resp. }
\lim_{t\to 0}\|w_2(\cdot, t) - g\|_{L^\infty} =0).   
\end{equation}
\end{theorem}
\begin{proof}
We will first check that $w_1$ and $w_2$ are well-defined, for this reason we begin with the first part of (b). We will rewrite $w_1$ and $w_2$ as a convolution of the data on $\bR_+$ with $\cE,$
\begin{equation}\label{r-rep}
w_1(x,t) = \int_{\bR}  \cE(x-y,t) \tilde g(y)\,dy, \qquad
w_2(x,t) = \int_{\bR}  \cE(x-y,t) \bar  g(y)\,dy ,
\end{equation}
where $\tilde g$ (resp. $\bar g$) is an odd extension, i.e. $\tilde g(-y) = - g(y)$ (resp. even extension, i.e. $\bar g(-y) = g(y)$) for $y>0.$
Since $\cE_t \in L^1(\bR)$ and $\int_\bR\cE_t \,dx =1$, then Young inequality for convolutions imply that $\cE_t * \tilde g, \cE_t * \bar g\in L^p(\bR)$, when $p<\infty$ and 
$$
\| \cE_t * \tilde g\|_{L^p(\bR)} \le 2 \| g\|_{L^p(\bR_+)},
\qquad \| \cE_t * \bar g\|_{L^p(\bR)} \le 2\| g\|_{L^p(\bR_+)}.
$$
When $g$ is bounded, then 
$$
\| \cE_t * \tilde g\|_{L^\infty} \le \| g\|_{L^\infty}, \qquad
\| \cE_t * \bar g\|_{L^\infty} \le \| g\|_{L^\infty}.
$$
We conclude that $w_1$ and $w_2$ are well-defined.

Let us argue that  $w_1$, $w_2$ are solutions to (\ref{rn3}). We will provide some details for $w_1$, (the proof for $w_2$ is the same). 
We first notice that the
kernel $\cE_t$ is a composition of an analytic function with $x\mapsto |x|^{1+\alpha}$, hence the convolution appearing in the definition of $w_i$, $i=1,2$ shares this kind of smoothness.
Since $w_1$ is a $C^{1+\alpha}$-function we may apply $D^\alpha_{x_C}$ to $v$ defined above. We obtain,
$$
D^\alpha_{x_C} w_1(x,t) = \frac1{\Gamma(1-\alpha)}\int_0^x \frac{dy}{(x-y)^\alpha}
\int_{-\infty}^\infty \frac \partial {\partial y} \cE_t (y+ s) \tilde g(s)\,dy,
$$
where we could interchange the integral over $\bR$ with the differentiation due to the integrability of $\frac \partial {\partial y} \cE_t (y+ s) \tilde g(s)$ with respect $s$ for all $y\in \bR$. Now, we notice that we may
invoke the Fubini Theorem to interchange the order of integrals. Thus, we conclude that
the partial integration operator $I^{1-\alpha}$ and the integration over $\bR_+$ with respect to $y$ commute and we see,
$$
D^\alpha_{x_C} w_1(x,t) = \int_{-\infty}^\infty  D^\alpha_{y_C} \cE_t (x+ y) \tilde g(y)\,dy.
$$
Now, due to regularity of $\cE_t$ we see that
$$
\frac \partial {\partial x} D^\alpha_{x_C} w_1= 
\int_{-\infty}^\infty   \frac\partial {\partial x}D^\alpha_{y_C} \cE_t (x+ s)\tilde g(s)\,dy.
$$
Thus, we conclude that $( \frac \partial {\partial t}  - \frac \partial {\partial x} D^\alpha_{x_C}) w_1 = 0$.  

Now, we check the boundary conditions. Since $w_1(\cdot,t)$ is continuous up to $x=0$ for $t>0$ we see that
$$
w_1(0,t) = \int_0^\infty ( \cE_t(-y) - \cE_t(y))g(y)\,dy =0.
$$
The RHS vanishes because $\cE $ is even.

The argument for $w_2$ is similar. We use that $\frac{\partial w_2}{\partial x}(\cdot, t)$ is continuous up to $\{x =0\}$ for $t>0$. For positive $x$ we have
$$
\frac{\partial w_2}{\partial x}(x,t)= -(1+\alpha) 
\int_0^\infty( \frac{\partial \cE_t}{\partial x}(x-y) \sgn(x-y)|x-y|^\alpha + 
\frac{\partial \cE_t}{\partial x}(x+y) \sgn(x-y)|x-y|^\alpha) g(y)\,dy.
$$
Thus, 
$$
\lim_{x\to0^+} \frac{\partial w_2}{\partial x}(x,t)= -(1+\alpha) 
\int_0^\infty( - \frac{\partial \cE_t}{\partial x}(-y) y^\alpha + 
\frac{\partial \cE_t}{\partial x}(y) y^\alpha) g(y)\,dy =0.
$$

Now, we turn our attention to the initial condition. We recall that (\ref{r-ic}) follows from the standard properties of convolution with a kernel whose integral is one.
%
\end{proof}
 
We present here convolution formulas to solve the non-homogeneous problem. After establishing them we may  say that indeed $\cE$ is a fundamental solution, because it behaves like one.
\begin{corollary}
Let us suppose that $f\in C^1([0,\infty)\times \bR_+)\cap C([0,\infty)^2)\cap L^\infty(\bR_+^2)$ and $f$ has a compact support. We set (in case of $w_3$ below we require that $f(0,t) =0$ for all $t>0$),
$$
w_3 (x,t) = \int_{0}^t \int_0^\infty (\cE(x-y, t-s) - \cE(x+y, t-s))f(y,s)\,dyds,
$$
$$
w_4 (x,t) = \int_{0}^t \int_0^\infty (\cE(x-y, t-s) + \cE(x+y, t-s))f(y,s)\,dyds.
$$
Then, for all $R>0$ we have $w_3,  w_4 \in C^{1+\alpha} ([0,R]\times(0,R])$ and $w_j(x,\cdot)\in C^\infty(0,\infty)$ for all $x>0$, $j=3,4$. Moreover, $w_3,  w_4$ are solutions to 
\begin{equation}\label{rn33}
\begin{array}{ll}\displaystyle{
\frac{\partial w }{\partial t} - \frac{\partial  }{\partial x} D^\alpha w = f}      & x,t >0, \\
 ~ & ~\\
w(x,0) = 0     & x>0.
\end{array}
\end{equation}
In addition, $w_3$ (resp. $w_4$) satisfies the Dirichlet (resp. Neumann) boundary conditions.
\end{corollary}
\begin{proof} We will present an argument for $w_3$. The proof for $w_4$ goes along the same lines.

We extend $f$ by odd reflection, $\tilde f(-y,t) = - f(-y,t)$, for $y>0$. Then, $w_3$ takes the following form, where we use the commutativity of the convolution,
$$
w_3(x,t) = \int_{0}^t \int_{\bR} \cE(x-y, t-s) \tilde f(y,s) \,dyds=
 \int_{0}^t \int_{\bR} \cE(y,s) \tilde f (x-y, t-s)\,dyds.
$$
Due to the compactness of the support of $\tilde f$, we have
$$
\frac {\partial w_3}{\partial t} (x,t) =  \int_{\bR} \cE(y,s) 
\tilde f (x-y, 0)\,dy +
\int_{0}^t \int_{\bR} \cE(y,s) 
\frac {\partial \tilde f}{\partial t} (x-y, t-s)\,dyds.
$$
Then, it is easy to establish by carefully employing the integration by parts that
$$
\frac {\partial w_3}{\partial t} (x,t) =  \int_{\bR} \cE(y,s) 
\tilde f (x-y, 0)\,dy +
\int_{0}^t \int_{\bR} \frac {\partial}{\partial s}\cE(y,s) 
 \tilde f (x-y, t-s)\,dyds.
$$
We have already seen in the course of proof of Theorem \ref{prop1} that 
$$
D^\alpha_{x_C}w_3(x,t) = \int_t\int_\bR D^\alpha_{x_C} \cE(x-y, t-s) \tilde f(y,s) \,dyds.
$$
Subsequently, it is easy to check that
$$
\frac \partial{\partial x} D^\alpha_{x_C} w_3(x,t)
=\int_t\int_\bR \frac \partial{\partial y}D^\alpha_{y_C} \cE f(y,s) \tilde (x-y, t-s) \,dyds.
$$
Hence, $w_3$ is a solution to (\ref{rn33}).

An argument used in the proof of Theorem \ref{prop1} shows that $w_3$ satisfies the Dirichlet boundary condition. Now, we shall investigate the initial condition. Due to the boundedness of $f$ we see that
$$
|w_3(x,t)| \le \int_0^t\int_\bR \cE(y,s) \| f\|_{L^\infty}\,dyds \le t \| f\|_{L^\infty}\,dyds.
$$
Our claims follow.
\end{proof}


We constructed solutions to (\ref{rn1}) with the help of the convolution of the fundamental solution with the data. Since this case is not covered neither in \cite{NaRy} nor in \cite{KR1} we have to show uniqueness separately. The difficulty with the classical method of testing the equation with the solution is that 
integrability of the derivatives (fractional and integer) of  $\cE$ is different than one might expect.  Here, we present an observation which turns out very useful.
\begin{lemma}\label{le-zanik}
Let us suppose that $\Phi$ is given by (\ref{def-phi}). Then, 
for all $x>0$ we have,
$$
0< x \,\Phi (x)\le 2.
$$
\end{lemma}\noindent
{\it Proof.}
We combine (\ref{ode}) and (\ref{r-oszac}) to obtain
$$
\Phi(0) - \Phi(x) = \frac 1{\Gamma(\alpha)} \int_0^x \frac{z \Phi(z)\, dz}{(x-z)^{1-\alpha}}.
$$
Taking into account the positivity of $\Phi$, we obtain for $x>1$ that
$$
1 = \Phi(0)  \ge \frac 1{\Gamma(\alpha)} \int_{x-1}^x \frac{z \Phi(z)\, dz}{(x-z)^{1-\alpha}}.
$$
Since Theorem \ref{thm-main} (3) guarantees monotonicity of $\Phi$, then we obtain the estimate
$$
\Gamma(\alpha)\ge (x-1)\Phi(x) \int_{x-1}^x \frac{dz}{(x-z)^{1-\alpha}}
= \frac 1\alpha (x-1)\Phi(x).
$$
Hence, for $x\in [0,2]$ we have $x\Phi (x) \le 2$. For $x>2$ we obtain
$$
x\Phi(x) \le \Gamma(1+\alpha) \frac x{x-1} \le 2. \eqno\qed
$$

We will use the above Lemma to show limited integrability of the derivatives of solutions constructed in Theorem \ref{prop1}.
\begin{lemma}\label{le-pi}
Let us suppose that $g\in W^{1,1}(0,\infty)$, $g(0)=0$ and $w_1$ is given by (\ref{rn3.5}) and  $w_2$ is given by (\ref{rn3.6}). We also assume that $g$ in the definition of $w_1$ satisfies $g(0) =0$. Then,\\
(1) $\frac \partial {\partial t} w_i (\cdot, t) \in L^1(0,\infty)$ for all $t>0$ and $i=1,2$;\\
(2) $D^\alpha_{x_C} w_i (\cdot, t) \in L^\infty(0,\infty)$ for all $t>0$ and $i=1,2$;\\
(3) $\frac \partial {\partial x} w_i (\cdot, t) \in L^1(0,\infty)$ for all $t>0$ and $i=1,2$.
\end{lemma}
\begin{proof}
We will use the representation formulas (\ref{r-rep}). The argument is conducted simultaneously for $\tilde g$ and $\bar g$. For the sake of simplicity of notation we will write here $g$ for both $\tilde g$ and $\bar g$ and also $w$ will denote $w_1$ and $w_2$.

We will check that (1) holds. 
Let us compute $\frac \partial {\partial t} w$, we will express it in term of $\Phi$ introduced in (\ref{def-phi}),
$$
\frac \partial {\partial t} w(x,t) = - \frac{a_0}{(1+\alpha) t^{1+ \frac 1{1+\alpha}}}
\int_{-\infty}^\infty \Phi \left( \frac{x-y}{t^{\frac1{1+\alpha}}}\right) g(y)\,dy
- \frac{a_0}{(1+\alpha) t^{1+\frac 2{1+\alpha}}}
\int_{-\infty}^\infty \frac{d\Phi }{d\xi} \left( \frac{x-y}{t^{\frac1{1+\alpha}}}\right) g(y)\,dy.
$$
The RHS above is well-defined because $g$ has a compact support and $\Phi$ is a $C^{1+\alpha}$-function. Since $g$ belongs to $W^{1,1}(\bR_+)$ we may integrate the last term by parts. Here, in the case of the odd extension of $g$ we use $g(0)=0$.

Finally, we reach
$$
\frac \partial {\partial t} w(x,t) = - \frac 1{(1+\alpha)t}
\int_0^\infty \cE_t(x-y)g(y)\,dy
- \frac1{(1+\alpha) t^{1+\frac 1{1+\alpha}}}
\int_0^\infty \cE_t(x-y) g'(y)\,dy.
$$
The integrability of $\cE$, $g$ and $g'$ implies claim (1) for $w$.


We are going to establish part (2). We have already seen that $D^\alpha_{x_C}$ commute with integration, so we have,
$$
D^\alpha_{x_C} w = \int_{-\infty}^\infty D^\alpha_{x_C} \cE_t(x-y) g(y)\, dy.
$$
Since $\cE_t(x) = \frac{a_0}{t^{\frac1{1+\alpha}}} 
\Phi(x t^{-\frac1{1+\alpha}})$, we have to calculate the Caputo derivative of a scaled function. Let us suppose that $f$ is absolutely continuous on $[0,\infty$). We set $f_\lambda(x) = f(\lambda x)$. We compute $D^\alpha_{x_C}f_\lambda$,
$$
D^\alpha_{x_C}f_\lambda(x) = \frac1{\Gamma(1-\alpha)}\int_0^x \frac{df_\lambda}{ds}(s)(x-s)^{-\alpha}\,ds
= \frac1{\Gamma(1-\alpha)}\int_0^x \lambda\frac{f'(\lambda s)}{(x-s)^\alpha}\,ds.
$$
After changing the variables $\lambda s = z$ we obtain,
\begin{equation}\label{r-scale}
D^\alpha_{x_C}f_\lambda(x) = 
\frac{\lambda^\alpha}{\Gamma(1-\alpha)} \int_0^{\lambda x} \frac{ f'(z)}{(\lambda x - z)^\alpha}\, dz
= \lambda^\alpha D^\alpha_{y_C}f(\lambda x).
\end{equation}
Taking  into account(\ref{r-scale})  yields,
$$
D^\alpha_{x_C} w =   \frac{a_0 }{t^{\frac 2{1+\alpha}}} \int_{-\infty}^\infty D^\alpha_{\xi_C} \Phi((x-y)t^{-\frac 1{1+\alpha}}) g(y)\, dy.
$$
Now, due to  (\ref{ode}) and Lemma \ref{le-zanik} we conclude that
$$
|D^\alpha_{x_C} v_2| = \left| a_0 t^{ - \frac1{1+\alpha}}
\int_{-\infty}^\infty \frac{(x - y)}{t^{\frac1{1+\alpha}}}  
\Phi\left(\frac{(x - y)}{t^{\frac1{1+\alpha}}}\right) g(y)\, dy \right|
\le a_0 2 t^{ - \frac1{1+\alpha}}.
$$
Part (2) follows.

Part (3) is established along the lines of the proof of (1). 
Let us compute the derivative of $w$, then we see
$$
\frac \partial {\partial x} w(x,t) =
\int_{-\infty}^\infty  \frac \partial {\partial x} \cE_t(x+y) g(y)\, dy
= - \int_{-\infty}^\infty  \cE_t(x+y) g'(y)\, dy,
$$
where we used  
the boundedness of the support of $g$.
\end{proof}

We would like to state our uniqueness result. For this purpose we define a class of functions, which we find suitable,
\begin{align*}
 X =& \,C([0,\infty)^2)\cap C^{1+\alpha}([0,\infty)\times(0,\infty))
\cap\\
&\{ u\in L^\infty(\bR_+; L^2(\bR_+)):\ \forall t>0\ u_t(\cdot, t), u_x(\cdot, t)\in L^1(\bR_+), 
\ D^\alpha_{x_c}u(\cdot, t) \in L^\infty(\bR_+)\}. 
\end{align*}

\begin{proposition}\label{pr-uq}
If $w$ is a solution to (\ref{rn3}) with either Dirichlet or Neumann boundary data (\ref{r12.5}) and $w\in X$, then $w$ is unique.
\end{proposition}
\begin{proof} 
We take the difference $w$ of two solutions $u_1$ and $u_2$ from $X$, we multiply them by $w$ and integrate over $\bR_+$. The definition of the class $X$ permits us to write
$$
\frac12 \frac d{dt} \|w(t)\|^2_{L^2} =\langle w_t, w \rangle \equiv \int_0^\infty \frac\partial{\partial x} D^\alpha_{x_C}w(x,t) w(x,t)\,dx. 
$$
Since $w(\cdot,t)$ is continuous on $\bR_+$ as well as it is in $L^2$, then there exists a sequence $R_n$ converging to infinity, such that $w(R_n,t) D^\alpha w(R_n,t) $ goes to zero, when $n\to \infty.$ Thus, after integration by parts over $[0,R_n]$ the RHS above takes the following form,
$$
\int_0^\infty \frac\partial{\partial x} D^\alpha_{x_C}w(x,t) w(x,t)\,dx
= \lim_{n\to\infty} 
\int_0^{R_n} \frac\partial{\partial x} D^\alpha_{x_C}w(x,t) w(x,t)\,dx =
 -\lim_{n\to\infty} 
\int_0^{R_n}  D^\alpha_{x_C}w(x,t)\frac\partial{\partial x} w(x,t)\,dx,
$$
where we also take into account the zero Dirichlet data at $x=0$ or vanishing $D^\alpha w(0)$. In order to estimate the RHS we recall that \cite[Proposition 6.10]{KuYa} implies that
$$
\int_0^{R_n}  D^\alpha_{x_C}w(x,t)\frac\partial{\partial x} w(x,t)\,dx \ge 0.
$$
See formula (5) in the proof of \cite[Theorem 1]{KR1} for more details. As a result, we conclude that
$$
\frac d{dt} \|w(t)\|^2_{L^2} \le 0.
$$
Hence, for all $t>0$ we have $\|w(t)\|_{L^2}^2 \le \|w(0)\|_{L^2}^2 =0$.
\end{proof}

It is interesting to check when a solution belongs to  class $X.$ We do not offer a full answer, however, Lemma \ref{le-pi} gives us a hint.
We note :
\begin{corollary}
Let us suppose that $f$ is in $C_c(\bR_+^2)$. Then,\\
(1) If  $g$ is in $W^{1,1}(\bR_+)$ with bounded support and $g(0) =0$, then $w_1 + w_3$ is a unique solution to (\ref{rn33}) with initial condition (\ref{rn3}${}_2$) and boundary condition (\ref{r12.5}${}_1$).\\
(2) If  $g$ is in $W^{1,1}(\bR_+)$ with bounded support, then $w_2 + w_4$ is a unique solution to (\ref{rn33}) with initial condition (\ref{rn3}${}_2$) and boundary condition (\ref{r12.5}${}_2$).
\end{corollary}
\begin{proof}
Lemma \ref{le-pi} shows that indeed $w_1$ and $w_2$ are in class $X$. The calculations we performed in the course of  proof of Theorem \ref{prop1} show that  $w_3$, $w_4$ also belong to $X$. Then, we use Proposition \ref{pr-uq} to finish the proof.
\end{proof}



Having established an integral representation of solution we may draw conclusions about their asymptotic behavior. Here we note a decay property.

\begin{proposition}
Let us suppose that the assumptions of the uniqueness theorem, Proposition \ref{pr-uq}, hold and $g\in W^{1,1}(\bR_+)$ has bounded support. If $u$ is a unique solution to (\ref{rn3}) corresponding solution to $g$, then
$$
\sup_{x\in\bR_+} |u(x, t)| \le C t^{-1/(1+\alpha)} \|g\|_{L^1}
$$
\end{proposition}
\begin{proof}
We use the representation formula and the boundedness of $\cE$.
\end{proof}

We proved in \cite{NaRy} that viscosity solutions depend continuously upon $\alpha\in [0,1].$ We can use the representation formula to establish the same result. Indeed, we can show:
\begin{proposition}
If $w^\alpha_i$ , $i=1,2$ then 
$$
\lim_{\alpha \to 1^-} w^\alpha_i = w^0_i, \qquad i =1,2,
$$
where $w^0_i$ is the solution to the heat equation.
\end{proposition}
\begin{proof}
We use here the uniform convergence of the Mittag-Leffler functions $E_{\alpha, 1+1/\alpha, 1/\alpha}$ to $E_{1,2,1}$. This follows from the uniform boundedness of the family $E_{\alpha, 1+1/\alpha, 1/\alpha}$ and the Montel Theorem.
\end{proof}
\begin{remark}
This proof breaks down if we try to pass to the limit as $\alpha$ goes to 0. In this case $E_{\alpha, 1+1/\alpha, 1/\alpha}$ converges to $\frac 1{1-z}$ for $|z|<1.$
\end{remark}

The observation made in the  Proposition above has a bit surprising consequences. It suggests that for small $\alpha$ we should see phenomena typical for the hyperbolic problems.

We could relate this observation to the behavior  the discretization scheme. We presented  in \cite{NRV} the $\frac12$-shifted Gr\"unwald approximation. We
write out the scheme, for the Dirichlet 
data, in terms of the Gr\"unwald weights. The approximation of $u(i\Delta x, k\Delta t)$ takes  the following form,
\begin{equation}\label{RLscheme}
\begin{aligned}
  & u_0^{k+1}=u_0^k, \qquad u_n^{k+1}=u_n^k,
\\
&u_i^{k+1}=\sum_{j=0}^{i-1}{\beta(g_{i+1-j}-g_{i-j})  u_{j}^k} + \left(1+\beta (g_1-g_0) \right)u_i^k +\beta g_0 u_{i+1}^k,
\end{aligned}
\end{equation}
where 
and the Gr\"unwald weights are given by
\begin{equation}
\label{gweights}
\beta = \frac{\Delta t}{(\Delta x)^{1+\alpha}},\quad
g_0 =1 , \quad
g_i =\frac{i-1-\alpha}{i}g_{i-1},\quad i =1,2, \ldots .
\end{equation}
We see that when $\alpha\to0$,  then  (\ref{RLscheme})--(\ref{gweights}) converge to an explicit finite difference scheme for the transport equation. 
This  suggests a finite speed of propagation. This is  supported by our simulations, presented there which showed an initial pulse tending toward the left.

There are two  sides of the same coin. For $\alpha =1$, eq. (\ref{rn1}) becomes a heat equation, where the speed of propagation is infinite. It means that if the initial perturbation is non-negative and it has a compact support, then solutions to (\ref{rn1}) will be positive everywhere for $t>0$. Interestingly, when  $\alpha\to1$, then  (\ref{RLscheme})--(\ref{gweights}) converges to an explicit finite difference scheme for the heat equation. We also noticed the smearing out effect,  see Fig. 2 in \cite{NRV}.

Actually, we can show:
\begin{proposition}\label{pi}
For all $\alpha\in (0,1]$ the speed of signal propagation for solutions to (\ref{rn1}) with initial condition $u(x,0)= f(x)$ and the zero Neumann data is infinite, i.e. if $f\ge 0$, $f\neq 0$ and $\supp{f}\subset(0,R)$, where $R>0$, then for all $x,t>0$ we have
$$
w_2(x,t)>0.
$$
\end{proposition}
\begin{proof}
This is an immediate consequence of the definition of $w_2$ and positivity of $\cE$ and $f$.
\end{proof}
We have chosen the  Neumann data, because of the simplicity of the formula for a solution. 

\section*{Appendix}
Here we present a different and more elegant approach to the derivation of $\cE$ and its positivity. It depends on the properties of the three-parameter generalized Mittag-Leffler function and the arguments are shorter. 

We first derive the form of $\cE$ assuming that $u$ defined as
$$
u(x,t) = t^{-\frac 1{1+\alpha}} v(x  t^{-\frac 1{1+\alpha}})
$$
is a solution to (\ref{rn1}). Then, Corollary (\ref{co-ald}) yields,
$$
(D^\alpha_x u)_x (x,t) = t^{-1 - \frac 1{1+\alpha}} (D^\alpha_y v)_y(y),
$$
where $y = x t^{-\frac1{1+\alpha}}.$ Moreover, it is easy to check directly that
$$
u_t(x,t) = - \frac1{1+\alpha} t^{-1 - \frac1{1+\alpha}}( v(y) + y v'(y)).
$$
Combining these observations with (\ref{rn1}) yields,
$$
- \frac1{1+\alpha} t^{-1 - \frac1{1+\alpha}}( v(y) + y v'(y)) = (D^\alpha_y v)_y(y).
$$
This can be rewritten as
$$
\frac{d}{dy}\left(D_y^\alpha v(y)+\frac{1}{1+\alpha}yv(y)\right)=0.
$$
Since we are interested in smooth solutions then, due to \cite[Proposition 3.1]{NRV} we see that
$$
 D_y^\alpha v(y)+\frac{1}{1+\alpha}yv(y)=0, \qquad y>0. \eqno(*)
$$
According to Proposition \ref{rnieE} 
a  solution to  this equation is given by the formula
\begin{equation*}
v(y)=v(0)E_{\alpha,1+1/\alpha,1/\alpha}\left(-\frac{y^{1+\alpha}}{1+\alpha}\right).
\end{equation*}
Here $E_{\alpha,1+1/\alpha,1/\alpha}$ is a generalized Mittag-Leffler function defined in (\ref{df-ML}).

We can also offer much shorter and easier proof of the positivity of $\cE$. It is based again on the theory of the  Mittag-Leffler function and ODEs.

\bigskip\noindent {\bf Lemma A.}
Function $(0, \infty) \ni x \mapsto v(x) \equiv E_{\alpha, 1+\frac{1}{\alpha}, \frac{1}{\alpha}}(-\frac{x^{1+\alpha}}{1+\alpha})$ is positive for all $x > 0$.
\bigskip

\begin{proof}
We consider eq. (*) 
for $v$ with the initial condition $v(0) = 1$
Then, $v(x) \equiv E_{\alpha, 1+\frac{1}{\alpha}, \frac{1}{\alpha}}(-\frac{x^{1+\alpha}}{1+\alpha})$ is a solution to (*). 
Let us suppose our claim is not valid and the set
\begin{equation*}
A := \{ x>0 : v(x) \leq 0 \}
\end{equation*}
is not empty. Due to the continuity of $v$, there is $x_0$ such that
\begin{equation*}
v(x_0) = 0\qquad\hbox{and}\qquad
v(x) >0 \quad \mbox{for} \,\, x \in (0, x_0).
\end{equation*}
However, we have
\begin{align*}
\displaystyle v(x_0) &= -\frac{1+\alpha}{x_0}D_{x}^{\alpha}v(x_0) \\
 &= -\frac{1+\alpha}{x_0 \Gamma(1-\alpha)} \left( \frac{v(x_0) - v(0)}{x_0^\alpha} + \alpha \int_0^{x_0} \frac{v(x_0) - v(x_0 - z)}{z^{\alpha + 1}} dz \right) \\
 &= \frac{1+\alpha}{x_0 \Gamma(1-\alpha)} \left( \frac{1}{x_0^\alpha} + \alpha \int_0^{x_0} \frac{v(x_0 - z)}{z^{\alpha + 1}} dz \right) > 0.
\end{align*}
This is a contradiction.
\end{proof}

\subsection*{Acknowledgments}
PR was in part supported by  the National Science Centre, Poland, through the grant number 2017/26/M/ST1/00700.
This research was initiated during a PR visit to the University of Tokyo in 2019, whose hospitality is greatly appreciated.


\begin{thebibliography}{99}
%
\bibitem{valdinoci}
S.Dipierro, B.Pellacci, E.Valdinoci, G.Verzini,
Time-fractional equations with reaction terms: fundamental solutions and asymptotics,  	{\it arXiv:1903.11939.}
%
\bibitem{goro}
R.Gorenflo, A.A.Kilbas, S.V. Rogosin,  On the generalized Mittag-Leffler type functions. {\it Integral Transform. Spec. Funct.} {\bf 7} (1998), no. 3-4, 215-224.
%
\bibitem{saigo}
A.A. Kilbas, M.Saigo, On mittag-leffler type function, fractional calculus operators
and solutions of integral equations, {\it Integral Transform. Spec. Funct.}, {\bf  4}  (1996) no 4, 355--370.
%
\bibitem{kilbas} A.A.Kilbas, H.M.Srivastava, J.J.Trujillo,  Theory and applications of fractional differential equations. North-Holland Mathematics Studies, 204. Elsevier Science B.V., Amsterdam, 2006.
%
\bibitem{kim}
K.-H.Kim, D.Park, J.Ryu, An $L^q(L^p)$-theory for diffusion equations with space-time nonlocal operators, {\it  J. Differential Equations}, {\bf 287} (2021), 376-427.
%
\bibitem{KuYa}
A.Kubica, M. Yamamoto, Initial-Boundary Value Problems For Fractional Diffusion Equations With Time-Dependent Coefficients, {\it Fract. Calc. Appl. Anal.}, {\bf  21} (2018),  276--311.
%
\bibitem{luczko}
Y.Luchko, F.Mainardi, Y.Povstenko, Propagation speed of the maximum of the fundamental solution to the fractional diffusion-wave equation, {\it Comput. Math. Appl.}, {\bf 66} (2013), no. 5, 774--784.
%
\bibitem{mainardi}
F.Mainardi, Y.Luchko, G.Pagnini, 
The fundamental solution of the space-time fractional diffusion equation, {\it
Fract. Calc. Appl. Anal.}, {\bf 4} (2001), no. 2, 153--192.
%
\bibitem{NaRy}T.Namba, P.Rybka, On viscosity solutions of space-fractional diffusion equations of Caputo type, {\it SIAM J. Math. Anal} {\bf 52}, (2020), no 1, 653--681.
%
\bibitem{NRV}
T.Namba, P.Rybka, V.Voller, Some comments on using fractional derivative operators in modeling non-local diffusion processes,  {\it J. Comput. Appl. Math.} {\bf  381} (2021), 113040.
%
\bibitem{Klekot}
Y.Povstenko, J.Klekot, 
Fractional heat conduction with heat absorption in a sphere under Dirichlet boundary condition,  {\it Comput. Appl. Math.} {\bf  37} (2018), no. 4, 4475--4483. 
%
\bibitem{pschu}
A.V.Pskhu, The fundamental solution of a diffusion-wave
equation of fractional order, {\it Izv. Math.} {\bf  73} (2009) 141--182.
%
\bibitem{roscani} S. D. Roscani, D. A. Tarzia and L. Venturato, The similarity method and explicit solutions for the fractional space  one-phase Stefan problems, {\it arXiv:2009.12977}
%
\bibitem{KR1} K.Ryszewska,  An analytic semigroup generated by a fractional differential operator, {\it J. Math. Anal. Appl.} {\bf  483} (2020), no. 2, 123654, 17 pp.
%
\bibitem{KR2} K.Ryszewska,  A space-fractional Stefan problem, {\it Nonlinear Anal.} {\bf  199} (2020), 112027, 30 pp.
%
%
\bibitem{VV}
V.R. Voller, On a fractional derivative form of the Green-Ampt infiltration model, {\it Adv. Water Res.,} {\bf  34} (2010),  257--262.
\end{thebibliography}
\end{document}